\documentclass[final,leqno,onefignum,onetabnum]{siamltex1213}
\usepackage{algorithm}
\usepackage{overpic}
\usepackage{eepic}
\usepackage{graphicx}
\usepackage{subfigure}
\usepackage{algorithmic}
\usepackage{array}
\usepackage{bm}
\usepackage{amsfonts}
\usepackage{amsmath}
\usepackage{amssymb}
\usepackage{mathrsfs}
\usepackage{algorithmic}
\usepackage{algorithm}
\usepackage{hyperref}

\title{On Accelerating the Regularized Alternating Least Square Algorithm for Tensors \thanks{This work was
supported by National Natural Science Foundation of China (Grants
No. 11401092), China Scholarship Council (Grants No.201406625025)}}

\author{Xiaofei Wang\thanks{School of Mathematics and Statistics, Northeast Normal University, Renmin Street 5268, Changchun, China
(\email{wangxf341@nenu.edu.cn})}\and
Carmeliza Navasca\thanks{Department of Mathematics, University of Alabama at Birmingham, 1300 University Boulevard, Birmingham, AL, USA
(\email{cnavasca@uab.edu})}\and
Stefan Kindermann\thanks{Industrial Mathematics Institute, Johannes Kepler Universitat Linz, Altenbergerstrasse 69, A-4040 Linz, Austria (kindermann@indmath.uni-linz.ac.at)}
}

\begin{document}
\maketitle
\slugger{simax}{simax}{xx}{x}{x--x}

\begin{abstract}
In this paper, we discuss the acceleration of
the regularized alternating least square (RALS) algorithm
for tensor approximation. We propose a fast iterative method using a Aitken-Stefensen like updates for the regularized algorithm.
Through numerical experiments, the fast algorithm demonstrate a faster convergence rate for the accelerated version  in comparison to both the standard and regularized alternating least squares algorithms. In addition, we analyze the global convergence based on
the Kurdyka-{\L}ojasiewicz inequality as well as show that the RALS algorithm has a linear local convergence rate.

\end{abstract}

\begin{keywords}
alternating least square, Kurdyka-{\L}ojasiewicz inequality, tensor approximation
\end{keywords}

\begin{AMS}
15A69, 65F30
\end{AMS}

\pagestyle{myheadings}
\thispagestyle{plain}

\section{Introduction}
Given a third order tensor $\mathcal{T}\in\mathbb{R}^{I\times J\times K}$,
we want to find the best approximation of $\mathcal{T}$ with $r$ rank-one components.
This tensor approximation can be posed as an optimization problem:\\
\hspace*{3.5cm}\text{minimize}\ \ $\frac{1}{2}\|\mathcal{T}-\sum\limits_{s=1}^r\mathbf{a}_s\circ\mathbf{b}_s\circ\mathbf{c}_s\|_F^2$\\
\hspace*{3.5cm}\text{subject to}\ \ $\mathbf{a}_s\in\mathbb{R}^I,\mathbf{b}_s\in\mathbb{R}^J,\mathbf{c}_s\in\mathbb{R}^K,~s=1,\cdots,r$\\
where $\mathbf{a}_s \circ\mathbf{b}_s \circ\mathbf{c}_s$ is a rank-one tensor generated by taking the outer products of three vectors, $\mathbf{a}_s, \mathbf{b}_s$ and $\mathbf{c}_s$.
A global minimizer of this objective function,
$\frac{1}{2}\|\mathcal{T}-\sum\limits_{s=1}^r\mathbf{a}_s\circ\mathbf{b}_s\circ\mathbf{c}_s\|_F^2$,
may not exist due to the ill-posedness \cite{silva,lim}
of low rank approximation, but developing algorithms to detect local minimizers or critical points of the objective function is
important for both theoretical research and practical application of tensor computations \cite{kolda}.

The conventional method, the alternating least-squares (ALS) algorithm \cite{carroll,harshman}, which was proposed 45 years ago remains the workhorse for
computing tensor approximations and decompositions. It is based on iteratively solving least-squares subproblems of the original nonlinear least-square objective functional using the Gauss-Seidel updating scheme. The subproblems are obtained through matricizing the given tensor and the rank-one tensor components. Under an assumption on the Hessian of the objective function, it has shown in \cite{uschmajew} that the ALS algorithm has a linear local convergence rate.
Despite the success of the ALS algorithm, it has some shortcomings \cite{comon,paatero}. The
non-uniqueness of the solution within the inner iterations of the ALS can substantially decrease the convergence rate.
This non-uniqueness can be avoided by introducing a
Tikhonov regularized term to the objective function \cite{lim,paatero}. However, this new update mechanism with a Tikhonov regularized term cannot guarantee
that the local minimizer is also a fixed point of the ALS update operator.
Another regularization \cite{navasca,li} was proposed to handle
the ALS algorithm by introducing a proximal term
into every subproblem instead of directly into the objective function.
This regularized version of the ALS algorithm is called the regularized alternating least-squares (RALS) algorithm.
It was shown in \cite{li} that
any limit point of every convergent subsequence from the RALS algorithm is
a critical point of the objective function.

Both of the ALS and RALS algorithms update one block of variables at each iteration while fixing other blocks.
Thus, these two algorithms can be considered under the framework of
several alternating block minimization techniques \cite{attouch,attouch1,xu}.
The Kurdyka-{\L}ojasiewicz inequality \cite{lojasiwicz} was the essential tool to show the global convergence of the ALS.
Attouch et. al. \cite{attouch,attouch1} study the convergence properties of alternating proximal
minimization algorithms for nonconvex structured functions. In \cite{xu},
Xu and Yin develop the block coordinate descent method with the Gauss-Seidel updating sweep for block multi-convex functions with applications to nonnegative tensor factorization and tensor completion. Instead of updating all the blocks in each loop as in  \cite{attouch,attouch1,xu}, an alternative approach is the maximum block improvement (MBI) method \cite{chen} which only update \emph{the maximally
improving} block per loop. In \cite{liz},  MBI was shown to handle tensor optimization models with spherical constraints.
Under some mild assumptions, Li et. al. \cite{liz} show that MBI has a global convergence and a linear local convergent rate.
Here we consider the convergence properties of regularized alternating least square (RALS) in the case the regularization parameter is static.
We show the global convergence of the RALS algorithm under the framework of
proximal alternating minimization \cite{attouch,attouch1}. The rate of this global convergence depends on the exponent of
the Kurdyka-{\L}ojasiewicz (KL) inequality. We show that the global convergence rate is either linear or sub-linear, but to further discern between the rates relies on a priori knowledge on the exponent of the KL inequality. In the appendix, we discuss the local theory of convergence of RALS, namely, when the sequence is close enough to a local minimizer, the RALS algorithm has a linear local convergence rate.

In this paper, we propose a new acceleration version of RALS by extending the Aitken-Stefensen acceleration formula in a matrix form. The corresponding numerical simulation results illustrate the effectiveness of our acceleration method. In addition, the new fast method outperforms
ALS and RALS with the Nesterov \cite{nesterov} accelerated updates.

%

This paper is organized as follows.
In Section 2, we introduce some notations and terminologies on the RALS algorithm for tensor approximation.
 In Section 3, we propose an acceleration version of algorithm.
The simulation experiment is shown in Section 4. In Section 5, we discuss the global convergence rates of the algorithm.
Finally, in Section 6 we summarize our conclusions and show some remaining problems of this work.

\section{The RALS algorithm for tensor approximation}
We focus on third-order tensors
$\mathcal{T}=(t_{ijk})\in\mathbb{R}^{I\times J\times K}$ with
three indices $1\leq i\leq I,1\leq j\leq J$ and $1\leq k\leq K$,
but all the methods proposed here can be applied to tensors of arbitrary $d$-th order. A third-order tensor $\mathcal{T}$ has column, row and tube fibers,
which are defined by fixing every index but one and
denoted by $\mathbf{t}_{:jk}$, $\mathbf{t}_{i:k}$ and $\mathbf{t}_{ij:}$ respectively.
Correspondingly, we obtain three matricizations of $\mathcal{T}$:
\begin{equation}
\begin{aligned}
\mathbf{T}_{(1)}=[\mathbf{t}_{:11},\cdots,\mathbf{t}_{:J1},\mathbf{t}_{:12},\cdots,\mathbf{t}_{:J2},\cdots,\mathbf{t}_{:1K},\cdots,\mathbf{t}_{:JK}],\\
\mathbf{T}_{(2)}=[\mathbf{t}_{1:1},\cdots,\mathbf{t}_{I:1},\mathbf{t}_{1:2},\cdots,\mathbf{t}_{I:2},\cdots,\mathbf{t}_{1:K},\cdots,\mathbf{t}_{I:K}],\\
\mathbf{T}_{(3)}=[\mathbf{t}_{11:},\cdots,\mathbf{t}_{I1:},\mathbf{t}_{12:},\cdots,\mathbf{t}_{I2:},\cdots,\mathbf{t}_{1J:},\cdots,\mathbf{t}_{IJ:}].\\
\end{aligned}
\end{equation}

The outer product $\mathbf{a}\circ\mathbf{b}\circ\mathbf{c}\in\mathbb{R}^{I\times J\times K}$ of three nonzero vectors
$\mathbf{a}\in\mathbb{R}^I, \mathbf{b}\in\mathbb{R}^J$ and $\mathbf{c}\in\mathbb{R}^K$ is
called a rank-one tensor with elements $a_ib_jc_k$ for all the indices.
A canonical polyadic (CP) decomposition of $\mathcal{T}\in\mathbb{R}^{I\times J\times K}$ expresses $\mathcal{T}$ as a sum of rank-one outer products:
\begin{equation}\label{cpd}
\mathcal{T}=\sum_{s=1}^{r} \mathbf{a}_s\circ\mathbf{b}_s\circ\mathbf{c}_s
\end{equation}
where $\mathbf{a}_s\in\mathbb{R}^I,\mathbf{b}_s\in\mathbb{R}^J,\mathbf{c}_s\in\mathbb{R}^K$ for $1\leq s\leq r$.
Every outer product $\mathbf{a}_s\circ\mathbf{b}_s\circ\mathbf{c}_s$ is a rank-one component. The positive integer $r$ is number of rank-one component number of tensor $\mathcal{T}$.

The Khatri-Rao product of two matrices $\mathbf{A}\in\mathbb{R}^{I\times r}$ and $\mathbf{B}\in\mathbb{R}^{J\times r}$
is defined as
$$\mathbf{A\odot B}=(\mathbf{a}_1\otimes\mathbf{b}_1,\cdots,\mathbf{a}_R\otimes\mathbf{b}_R)\in\mathbb{R}^{IJ\times r},$$
where the symbol ``$\mathbf{\otimes}$" denotes the Kronecker product:
$$\mathbf{a\otimes b}=(a_1b_1,\cdots,a_1b_J,\cdots,a_Ib_1,\cdots,a_Ib_J)^T.$$
Using this Khatri-Rao product, the CP decomposition (\ref{cpd}) can be written in three matricization forms of tensor $\mathcal{T}$:
\begin{equation}
\mathbf{T}_{(1)}=\mathbf{A}(\mathbf{C\odot B})^T, \mathbf{T}_{(2)}=\mathbf{B}(\mathbf{C\odot A})^T, \mathbf{T}_{(3)}=\mathbf{C}(\mathbf{B\odot A})^T
\end{equation}
where $\mathbf{A}=(\mathbf{a}_1,\cdots,\mathbf{a}_r)\in\mathbb{R}^{I\times r},\mathbf{B}=(\mathbf{b}_1,\cdots,\mathbf{b}_r)\in\mathbb{R}^{J\times r}$ and $\mathbf{C}=(\mathbf{c}_1,\cdots,\mathbf{c}_r)\in\mathbb{R}^{K\times r}$
are called the factor matrices of tensor $\mathcal{T}$.

Let $\mathscr{X}=\mathbb{R}^{I\times r}\times\mathbb{R}^{J\times r}\times\mathbb{R}^{K\times r}$ where $r$ is any given positive integer, the elements of $\mathscr{X}$ is denoted by $\mathbf{x}=(\mathbf{A},\mathbf{B},\mathbf{C})$, where $\mathbf{A}\in\mathbb{R}^{I\times r},\mathbf{B}\in\mathbb{R}^{J\times r},\mathbf{C}\in\mathbb{R}^{K\times r}$.
Note that $\mathbf{x}$ can be also viewed as a vector in $\mathbb{R}^{r(I+J+K)}$.
Given a tensor $\mathcal{T}\in\mathbb{R}^{I\times J\times K}$,
we consider its approximation by using the sum of $r$ rank-one components
$\sum\limits_{s=1}^r\mathbf{a}_s\circ\mathbf{b}_s\circ\mathbf{c}_s$,
and define a residual function $f:\mathscr{X}\rightarrow\mathbb{R}$ by
\begin{equation}
f(\mathbf{x})=f(\mathbf{A},\mathbf{B},\mathbf{C})\rightarrow
\frac{1}{2}\|\mathcal{T}-\sum\limits_{s=1}^r\mathbf{a}_s\circ\mathbf{b}_s\circ\mathbf{c}_s\|_F^2,
\end{equation}
where vectors $\mathbf{a}_s,\mathbf{b}_s,\mathbf{c}_s$ are columns of $\mathbf{A,B}$ and $\mathbf{C}$ respectively,
and $\|\cdot\|_F$ is the tensor Frobenius norm.
There may exist a local minimizer $\mathbf{x}^*=(\mathbf{A}^*,\mathbf{B}^*,\mathbf{C}^*)$
of $f(\mathbf{A},\mathbf{B},\mathbf{C})$, which is also a critical point of $f(\mathbf{x})$ such that $\nabla f(\mathbf{x}^*)=0$
since $f$ is a polynomial function.
Denote $\sum\limits_{s=1}^r\mathbf{a}^*_s\otimes\mathbf{b}^*_s\otimes\mathbf{c}^*_s$
as an optimal approximation of tensor $\mathcal{T}$ with rank at most $r$,
where vectors $\mathbf{a}_s^*,\mathbf{b}_s^*,\mathbf{c}_s^*$ are columns of $\mathbf{A}^*,\mathbf{B}^*$ and $\mathbf{C}^*$ respectively.

The approximation of a given tensor is implemented by the alternating least squares (ALS) algorithm.
Given a starting point $\mathbf{x}^{(0)}=(\mathbf{A}^{(0)},\mathbf{B}^{(0)},\mathbf{C}^{(0)})$,
we solve three subproblems iteratively:

\begin{equation}
\label{ALS}
\begin{split}
\mathbf{A}^{(n+1)}&=\mathop{\arg\min}\limits_{\mathbf{A}\in\mathbb{R}^{I\times r}}f(\mathbf{A},\mathbf{B}^{(n)},\mathbf{C}^{(n)})
=\mathop{\arg\min}\limits_{\mathbf{A}\in\mathbb{R}^{I\times r}}\frac{1}{2}\|\mathbf{T}_{(1)}-\mathbf{A}(\mathbf{C}^{(n)}\odot{\mathbf{B}^{(n)}}^T)\|_F^2,\\
\mathbf{B}^{(n+1)}&=\mathop{\arg\min}\limits_{\mathbf{B}\in\mathbb{R}^{J\times r}}f(\mathbf{A}^{(n+1)},\mathbf{B},\mathbf{C}^{(n)})
=\mathop{\arg\min}\limits_{\mathbf{B}\in\mathbb{R}^{J\times r}}\frac{1}{2}\|\mathbf{T}_{(2)}-\mathbf{B}(\mathbf{C}^{(n)}\odot{\mathbf{A}^{(n+1)}}^T)\|_F^2,\\
\mathbf{C}^{(n+1)}&=\mathop{\arg\min}\limits_{\mathbf{C}\in\mathbb{R}^{K\times r}}f(\mathbf{A}^{(n+1)},\mathbf{B}^{(n+1)},\mathbf{C})
=\mathop{\arg\min}\limits_{\mathbf{C}\in\mathbb{R}^{K\times r}}\frac{1}{2}\|\mathbf{T}_{(3)}-\mathbf{C}(\mathbf{B}^{(n+1)}\odot{\mathbf{A}^{(n+1)}}^T)\|_F^2.
\end{split}
\end{equation}

If every optimization problem possesses a unique solution, then one loop of (\ref{ALS})
defines an operator $S_{ALS}(\cdot)$ \cite{uschmajew} via
\begin{equation}\label{SALS}
(\mathbf{A}^{(n+1)},\mathbf{B}^{(n+1)},\mathbf{C}^{(n+1)})=\mathbf{x}^{(n+1)}
=S_{ALS}(\mathbf{x}^{(n)})=S_{ALS}(\mathbf{A}^{(n)},\mathbf{B}^{(n)},\mathbf{C}^{(n)}),
\end{equation}
where three matrices
\begin{equation}
\label{ESALS}
\begin{split}
\mathbf{A}^{(n+1)}&=(\mathbf{T}_{(1)}(\mathbf{C}^{(n)}\odot\mathbf{B}^{(n)}))((\mathbf{C}^{(n)}\odot\mathbf{B}^{(n)})^T(\mathbf{C}^{(n)}\odot\mathbf{B}^{(n)}))^{-1},\\
\mathbf{B}^{(n+1)}&=(\mathbf{T}_{(2)}(\mathbf{C}^{(n)}\odot\mathbf{A}^{(n+1)}))((\mathbf{C}^{(n)}\odot\mathbf{A}^{(n+1)})^T(\mathbf{C}^{(n)}\odot\mathbf{A}^{(n+1)}))^{-1},\\
\mathbf{C}^{(n+1)}&=(\mathbf{T}_{(3)}(\mathbf{B}^{(n+1)}\odot\mathbf{A}^{(n+1)}))((\mathbf{B}^{(n+1)}\odot\mathbf{A}^{(n+1)})^T(\mathbf{B}^{(n+1)}\odot\mathbf{A}^{(n+1)}))^{-1}
\end{split}
\end{equation}
are the least square solutions of (\ref{ALS}). Note that the inversion in (\ref{ESALS}) may not exist due to collinearity of the columns in the factor matrices, thus, we consider the generalized Moore-Penrose inverse in this case.

Since the computations in steps (\ref{ALS}) may not give a unique solution,
an extra regularized term (\cite{li,navasca}) is added in every step for eliminating the possibility of a non-uniqueness solution.
This regularized ALS algorithm (RALS) is shown as follows:
\begin{eqnarray}\label{RALS}
\mathbf{A}^{(n+1)}&=&\mathop{\arg\min}\limits_{\mathbf{A}\in\mathbb{R}^{I\times r}}f(\mathbf{A},\mathbf{B}^{(n)},\mathbf{C}^{(n)})+\frac{1}{2}\lambda\|\mathbf{A}-\mathbf{A}^{(n)}\|_F^2,\nonumber \\
\mathbf{B}^{(n+1)}&=&\mathop{\arg\min}\limits_{\mathbf{B}\in\mathbb{R}^{J\times r}}f(\mathbf{A}^{(n+1)},\mathbf{B},\mathbf{C}^{(n)})+\frac{1}{2}\lambda\|\mathbf{B}-\mathbf{B}^{(n)}\|_F^2,\\
\mathbf{C}^{(n+1)}&=&\mathop{\arg\min}\limits_{\mathbf{C}\in\mathbb{R}^{K\times r}}f(\mathbf{A}^{(n+1)},\mathbf{B}^{(n+1)},\mathbf{C})+\frac{1}{2}\lambda\|\mathbf{C}-\mathbf{C}^{(n)}\|_F^2,\nonumber
\end{eqnarray}
where $\lambda>0$ is a regularization parameter.
Our work is based on this RALS model and
addresses the case when the regularization parameter $\lambda$ is static.
It is easy to check that every subproblem in (\ref{RALS}) must have a
unique solution because of the strict convexity of the subproblem.
We denote the update of (\ref{RALS}) for $\mathbf{A,B,C}$ by using an operator $S(\cdot)$:
\begin{equation}\label{SRALS}
(\mathbf{A}^{(n+1)},\mathbf{B}^{(n+1)},\mathbf{C}^{(n+1)})=\mathbf{x}^{(n+1)}
=S(\mathbf{x}^{(n)})=S(\mathbf{A}^{(n)},\mathbf{B}^{(n)},\mathbf{C}^{(n)}),
\end{equation}
where three matrices

\begin{equation}\label{ESRALS}
\begin{split}
\mathbf{A}^{(n+1)}&=(\mathbf{T}_{(1)}(\mathbf{C}^{(n)}\odot\mathbf{B}^{(n)})+
\lambda\mathbf{A}^{(n)})((\mathbf{C}^{(n)}\odot\mathbf{B}^{(n)})^T(\mathbf{C}^{(n)}\odot\mathbf{B}^{(n)})+\lambda\mathbf{I})^{-1}, \\
\mathbf{B}^{(n+1)}&=(\mathbf{T}_{(2)}(\mathbf{C}^{(n)}\odot\mathbf{A}^{(n+1)})+
\lambda\mathbf{B}^{(n)})((\mathbf{C}^{(n)}\odot\mathbf{A}^{(n+1)})^T(\mathbf{C}^{(n)}\odot\mathbf{A}^{(n+1)})+\lambda\mathbf{I})^{-1},\\
\mathbf{C}^{(n+1)}&=(\mathbf{T}_{(3)}(\mathbf{B}^{(n+1)}\odot\mathbf{A}^{(n+1)})+
\lambda\mathbf{C}^{(n)})((\mathbf{B}^{(n+1)}\odot\mathbf{A}^{(n+1)})^T(\mathbf{B}^{(n+1)}\odot\mathbf{A}^{(n+1)})+\lambda\mathbf{I})^{-1}
\end{split}
\end{equation}

are the least square solutions of (\ref{RALS}).

The RALS algorithm can be viewed as a proximal regularization of a three block
Gauss-Seidel method for minimizing $f(\mathbf{A},\mathbf{B},\mathbf{C})$.
In the next section, we will show the global convergence of the RALS algorithm under the framework of
proximal alternating minimization \cite{attouch,bolte}.

\section{Acceleration of the RALS algorithm}
In this section, we suggest an acceleration technique for the RALS algorithm.
Our acceleration method is loosely based on the Aitken-Stefensen formula \cite{isaacson},
which is a conventional acceleration technique for numerical computation.
For a given convergent sequence $\{x^{(n)}\}_{n\in\mathbb{N}}$,
a new sequence $\{y^{(n)}\}_{n\in\mathbb{N}}$ is generated by
\begin{equation}\label{aitken}
y^{(n)}=x^{(n)}-\frac{(\bigtriangleup x^{(n)})^2}{\bigtriangleup^2x^{(n)}},
\end{equation}
where $\bigtriangleup x^{(n)}=x^{(n+1)}-x^{(n)}$ and $\bigtriangleup^2x^{(n)}=x^{(n+2)}-2x^{(n+1)}+x^{(n)}$.
For fixed point iteration,
the Aitken-Steffensen acceleration (\ref{aitken})
can achieve quadratic convergent rate \cite{isaacson} without requiring derivative terms.

The generalization of the Aitken-Stefensen process to a $k$-dimensional sequence requires
the following iterative formula:
\begin{equation}\label{aitken1}
\mathbf{y}^{(n)}=\mathbf{x}^{(n)}-\triangle\mathbf{X}^{(n)}(\triangle^2\mathbf{X}^{(n)})^{-1}\triangle\mathbf{x}^{(n)},
\end{equation}
where $\triangle\mathbf{x}^{(n)}=\mathbf{x}^{(n+1)}-\mathbf{x}^{(n)},
\triangle\mathbf{X}^{(n)}=(\mathbf{x}^{(n+1)}-\mathbf{x}^{(n)},\cdots,\mathbf{x}^{(n+k)}-\mathbf{x}^{(n+k-1)})$ and $\triangle^2\mathbf{X}^{(n)}=(\mathbf{x}^{(n+2)}-2\mathbf{x}^{(n+1)}+\mathbf{x}^{n},\cdots,\mathbf{x}^{(n+k+1)}-2\mathbf{x}^{(n+k)}+\mathbf{x}^{(n+k-1)})$.
The formula (\ref{aitken1}) for $\{\mathbf{y}^{(n)}\}_{n\in\mathbb{N}}$ also has a quadratic convergence rate
under five basic assumptions \cite{noda}.
Although the Aitken-Stefensen process for $k$-dimensional sequence theoretically has a fast convergent rate,
it has two main drawbacks in the practical implementation.
One is that to compute $\mathbf{y}^{(n)}$, an a priori set of sequences is needed, namely, $\mathbf{x}^{(1)}$ to $\mathbf{x}^{(n+k+1)}$.
Once the dimension $k$ of vector is large, the practical implement will be time-consuming especially when facing a complicated updating map.
The other is that this iterative process may be invalid
if the original sequence $\{\mathbf{x}^{(n)}\}_{n\in\mathbb{N}}$ converges fast and the dimension $k$ is large enough such that
$\mathbf{x}^{(n+k+1)}-2\mathbf{x}^{(n+k)}+\mathbf{x}^{(n+k-1)}$ is close to zero and $\triangle^2\mathbf{X}^{(n)}$ is singular.
So although the Aitken-Stefensen method can be directly
applied to the acceleration of the $r(I+J+K)$-dimensional sequence $\{\mathbf{x}^{(n)}\}_{n\in\mathbb{N}}$ generated by the RALS algorithm,
it does not work well especially when $I,J,K,r$ are large.
For example, if $I=J=K=20$ and $r=10$, then the dimension $k$ is 600.
To compute the initial vector of $\mathbf{y}^{(0)}$ from $\mathbf{x}^{(0)}$,
we need to know $601$ vectors from $\mathbf{x}^{(1)}$ to $\mathbf{x}^{(601)}$.
But the original sequence $\{\mathbf{x}^{(n)}\}_{n\in\mathbb{N}}$ from the RALS
may have already converged before $n=601$.

To obviate these drawbacks of the recursive formula (\ref{aitken1}) of vectors,
we utilize the matrix format of the update (\ref{ESRALS}) for the RALS algorithm
and propose a matrix based Aitken-Stefensen acceleration formula.
We denote the $(I+J+K)\times r$ matrix $({\mathbf{A}^{(n)}}^T,{\mathbf{B}^{(n)}}^T,{\mathbf{C}^{(n)}}^T)^T$ by $\mathbf{X}^{(n)}$,
and set the update by
\begin{equation}\label{aitken2}
\mathbf{X}_*^{(n+1)}=\mathbf{X}^{(n)}-\mathbf{Z}^{(n)},
\end{equation}
where $\mathbf{Z}^{(n)}$ is a solution of a linear system
\begin{equation}\label{linearsys}
\mathbf{Z}^{(n)}(S(S(\mathbf{X}^{(n)}))-2S(\mathbf{X}^{(n)})+\mathbf{X}^{(n)})^T= (S(\mathbf{X}^{(n)})-\mathbf{X}^{(n)})(S(\mathbf{X}^{(n)})-\mathbf{X}^{(n)})^T.
\end{equation}
Here the matrix $\mathbf{Z}^{(n)}$ can be understood as a small perturbation from $\mathbf{X}^{(n)}$
to $\mathbf{X}_*^{(n+1)}$ since $\|S(\mathbf{X}^{(n)})-\mathbf{X}^{(n)}\|_F^2$ is small
when $\mathbf{X}^{(n)}$ is close to a fixed point of $S$ (as defined by (\ref{SRALS})).
Note that $S(\mathbf{X}^{(n)})$ is based on the RALS, and
we denote the new update (\ref{aitken2}) from $\mathbf{X}^{(n)}$ to $\mathbf{X}_*^{(n+1)}$ by an operator $T$:
\begin{equation}
\mathbf{X}_*^{(n+1)}=T(\mathbf{X}^{(n)}).
\end{equation}
It can be verified that a fixed point of operator $T$
is also a fixed point of operator $S$.

Notice that besides one extra update from $S(\mathbf{X}^{(n)})$ to $(S(S(\mathbf{X}^{(n)}))$,
the formula (\ref{aitken2}) involves solving a large linear system (\ref{linearsys})
with the coefficient matrix $(S(S(\mathbf{X}^{(n)}))-2S(\mathbf{X}^{(n)})+\mathbf{X}^{(n)})^T$ of size $r\times(I+J+K)$.
If (\ref{aitken2}) is computed in each step of algorithm,
the whole time cost of the practical implement will be very huge.
So in the following Algorithm 1, we implement the formula (\ref{aitken2}) not at every step $n$,
but choose the implementation step $n$ with a fixed interval after the residual is small enough.
From another perspective, the formula allows the outer iteration of the (R)ALS algorithm to \emph{jump out}
from the linear convergent regions.
The residual gap generated by these perturbations can be quickly eliminated by
a fast decreasing speed.
Several numerical experiments are shown in the next section.

\begin{algorithm}[htb]
\renewcommand{\algorithmicrequire}{\textbf{Input:}}
\renewcommand\algorithmicensure {\textbf{Output:} }
\caption{Acceleration of RALS (RALS-A)}
\label{alg:Framwork}
\begin{algorithmic}[1]
\REQUIRE A third order tensor $\mathcal{T}\in\mathbb{R}^{I\times J\times K}$, the number $r$ of rank-one components,
an interval positive integer $q$, and a upper bound $\alpha\in\mathbb{R}$;\\
\ENSURE Three matrices $\mathbf{A}\in\mathbb{R}^{I\times r},\mathbf{B}\in\mathbb{R}^{J\times r},\mathbf{C}\in\mathbb{R}^{K\times r}$;\\
\STATE Give initial matrices $(\mathbf{A}^{(0)},\mathbf{B}^{(0)},\mathbf{C}^{(0)})$ and
let $\mathbf{X}^{(0)}=({\mathbf{A}^{(0)}}^T,{\mathbf{B}^{(0)}}^T,{\mathbf{C}^{(0)}}^T)^T$ and
set the error square as $\text{err}=\alpha$ .
\STATE Update step: \\
\STATE\textbf{for} $n=1,\cdots$ \textbf{do}\\
\STATE\hspace*{.5cm}\textbf{if} $\text{err}<\alpha$ and $n$ mod $q=0$ \textbf{do} \\
\STATE\hspace*{.5cm}\hspace*{.5cm}Compute matrices $S(S(\mathbf{X}^{(n)}))$ and $S(\mathbf{X}^{(n)})$ from $\mathbf{X}^{(n)}$.\\
\STATE\hspace*{.5cm}\hspace*{.5cm}Compute the matrix $\mathbf{X}_*^{(n+1)}$ by using (\ref{aitken2}).\\
\STATE\hspace*{.5cm}\hspace*{.5cm}$\mathbf{X}^{(n+1)}=\mathbf{X}_*^{(n+1)}$.\\
\STATE\hspace*{.5cm}\textbf{else do}\\
\STATE\hspace*{.5cm}\hspace*{.5cm}Compute the matrix $S(\mathbf{X}^{(n)})$ from $\mathbf{X}^{(n)}$.\\
\STATE\hspace*{.5cm}\hspace*{.5cm}$\mathbf{X}^{(n+1)}=S(\mathbf{X}^{(n)})$.\\
\STATE\hspace*{.5cm}\textbf{end if}
\STATE\hspace*{.5cm}$\text{err}=\|\mathbf{X}^{(n+1)}-\mathbf{X}^{(n)}\|_F^2$.
\STATE\textbf{end for}
\STATE $\mathbf{A}=\mathbf{A}^{(n)},~\mathbf{B}=\mathbf{B}^{(n)},~\mathbf{C}=\mathbf{C}^{(n)}$.\\
\RETURN  Three matrices $\mathbf{A},\mathbf{B}$ and $\mathbf{C}$.
\end{algorithmic}
\end{algorithm}

\section{Numerical experiments}
In this section we demonstrate the simulation experiments of the ALS, RALS algorithms and their accelerated versions.
Experiments are written in Matlab codes and implemented on a desktop computer with Intel i5 CPU 3.3GHz and 8G
memory. All of these algorithms are set to a tolerance error of $1\times 10^{-12}$ as a stopping criterion of
$$\|\mathbf{X}^{(n)}-\mathbf{X}^{(n-1)}\|_F^2=\|\mathbf{A}^{(n)}-\mathbf{A}^{(n-1)}\|_F^2+\|\mathbf{B}^{(n)}-\mathbf{B}^{(n-1)}\|_F^2
+\|\mathbf{C}^{(n)}-\mathbf{C}^{(n-1)}\|_F^2$$
between two subsequent iterates.
Algorithm \ref{alg:Framwork} is an accelerated version of the RALS algorithm, and we call it RALS-A.
We can similarly obtain an acceleration of the ALS algorithm; we call it ALS-A.
More specifically, the ALS-A can be obtained by replacing the update operator $S$ in Algorithm \ref{alg:Framwork}
by the operator $S_{ALS}$ in (\ref{SALS}).
The upper bound $\alpha$ is an input parameter for judging whether the original sequence is already in a linear convergent region.
While $\text{err}<\alpha$, we consider to implement the acceleration update in steps of a fixed interval $q$.
In the simulation experiments, we choose $\alpha=1\times 10^{-6}$ and $q=100$.
Except our acceleration way, we also consider the Nesterov-type updating way (RALS-Nes):
\begin{eqnarray*}\label{nesterov}
\mathbf{x}^{(n+1)}&=&S(\mathbf{x}_*^{(n)}),\\
\mathbf{x}_*^{(n+1)}&=&(1-\gamma_n)\mathbf{x}^{(n+1)}+\gamma_n\mathbf{x}^{(n)}
\end{eqnarray*}
where $\gamma_n=\frac{1-\mu_n}{\mu_{n+1}},\mu_n=\frac{1+\sqrt{1+4\mu_{n-1}^2}}{2},\mu_0=0$.
We can similarly obtain a Nesterov-type acceleration of the ALS algorithm (ALS-Nes) by replacing the update operator $S$
by the operator $S_{ALS}$.

First we consider time costs of the ALS, ALS-A, ALS-Nes, RALS, RALS-L, RALS-A and RALS-AL algorithms,
where the RALS-L and RALS-AL are two modified versions of the RALS and RALS-A with
a monotonically decreasing regularization parameter $\lambda$
to zero as the iteration number $n\rightarrow\infty$.
The rank-one component number $r$ is set to $10$ and dimensions $I=J=K$.
For each $I=10,20,50$, we do $100$ numerical experiments for these seven algorithms,
and record the corresponding seven medians of time costs on seconds.
As shown in Table \ref{tab}, the acceleration ALS-A and
RALS-A versions perform much better than the original ALS and RALS algorithms.
The RALS-L with decreasing $\lambda$ has a faster speed that the RALS,
and the RALS-AL has the fastest speed in all the algorithms basing on RALS.
The ALS-Nes consumes more times than other algorithms.
The reason may lie in that the Nesterov-type acceleration is designed for convex optimization \cite{beck,nesterov}. The main objective function of the RALS is a non-convex function while the subproblems are convex.

\begin{table}
\centering
\caption{Time costs of ALS, ALS-A, RALS,RALS-A, RALS-L and RALS-AL.}\label{tab}
    \begin{tabular}{| l | l | l | l | l | l | l | l | l |}
    \hline
    \tiny Algorithm &\tiny ALS &\tiny ALS-Nes  &\tiny ALS-A  &\tiny RALS &\tiny RALS-Nes &\tiny RALS-A &\tiny RALS-L &\tiny RALS-AL \\ \hline
     $I=10$ & 0.59  & 2.41 & 0.38 & 0.89 & 1.77  & 0.51 & 0.59 & \textbf{0.36}\\ \hline
     $I=20$ & 0.47 & 1.20 & 0.33 & 0.55  & 1.17 & 0.37 & 0.50 & \textbf{0.31} \\ \hline
     $I=50$ & 2.31 & 7.25 & \textbf{1.64} & 2.57 & 6.73 & 1.86 & 2.55 & 1.86 \\ \hline
    \end{tabular}
\end{table}

Second we consider the convergence of the ALS, ALS-A, RALS and RALS-A algorithms.
Two experiments are shown in Figure \ref{fig} according to the appearance of swamps of ALS or not.
In each experiment, $I=J=K=r=10$ and all of those algorithms use a same tensor
$\mathcal{T}\in\mathbb{R}^{10\times 10\times 10}$
with same initial factor matrices.
For the RALS and RALS-A algorithms, the regularization parameter $\lambda$ is fixed to $1$.
The plots in Figure \ref{fig} show the error square
$\|\mathbf{X}^{(n)}-\mathbf{X}^{(n-1)}\|_F^2$ versus the number $n$ of
iterations. As one can see, the convergence of the RALS algorithm is linear (see Appendix A),
and the acceleration version RALS-A has a faster convergent rate than the RALS.
This is similar for the ALS and ALS-A algorithms.
Notice that the ALS without swamps performs much better than the RALS with a fixed $\lambda$.
But as demonstrated in the following experiments,
the RALS algorithm with a decreasing $\lambda$ has a faster speed; see Table $1$.

\begin{figure*}
\centering
\subfigure[ALS with swamp.] {\includegraphics[height=2.7in,width=3.5in]{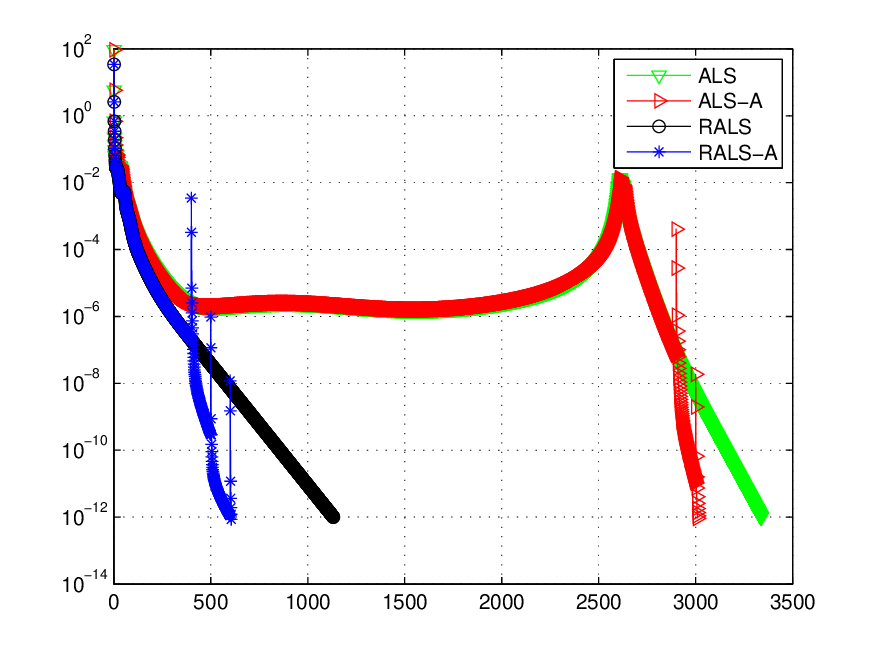}}
\subfigure[ALS without swamp.] {\includegraphics[height=2.7in,width=3.5in]{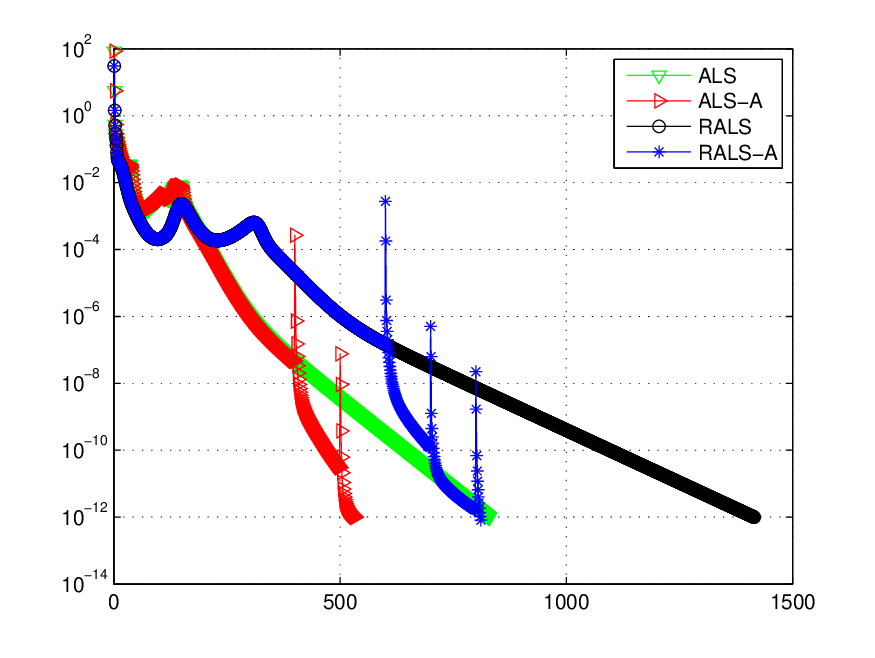}}
\caption{A comparison among ALS, ALS-A, RALS, and RALS-A.}\label{fig}
\end{figure*}

\section{Global convergence of RALS}
To discuss the global convergence of RALS, we need the Kurdyka-{\L}ojasiewicz inequality for real-analytic functions.
As shown in \cite{lojasiwicz}, we have the following proposition on the gradient inequality.

\begin{proposition}(Kurdyka-{\L}ojasiewicz Inequality)\label{KL}
Let $f(\mathbf{x})$ be a real-analytic function in a neighborhood
of $\mathbf{0}\in\mathbb{R}^n$ such that $f(\mathbf{0})=0$,
then the following inequality holds for some $0<\theta<1$
\begin{equation}
|f(\mathbf{x})|^{\theta}\leq\|\nabla f(\mathbf{x})\|
\end{equation}
in a neighborhood of $\mathbf{0}$.
\end{proposition}

Furthermore, if $f$ is a real-analytic function in a neighborhood of $\mathbf{a}\in\mathbb{R}^n$,
thus $g(\mathbf{x})=f(\mathbf{a}+\mathbf{x})-f(\mathbf{a})$ is a real-analytic
function in a neighborhood of $\mathbf{0}\in\mathbb{R}^n$ and $g(\mathbf{0})=0$.
From this Proposition \ref{KL}, we have that
$|f(\mathbf{a}+\mathbf{x})-f(\mathbf{a})|^\theta\leq\|\nabla f(\mathbf{a}+\mathbf{x})\|$
for any $\mathbf{x}$ in a neighborhood of $\mathbf{0}$.
It also follows that $|f(\mathbf{x})-f(\mathbf{a})|^\theta\leq\|\nabla f(\mathbf{x})\|$ for any $\mathbf{x}$ in a neighborhood of $\mathbf{a}$.
So we obtain another proposition as follows.

\begin{proposition}\label{KL2}
Let $f(\mathbf{x})$ be a real-analytic function on $\mathbb{R}^n$.
For any $\mathbf{a}\in\mathbb{R}^n$, there exists a real number $0<\theta<1$ and a neighborhood $U$ of $\mathbf{a}$ such that
\begin{equation}
|f(\mathbf{x})-f(\mathbf{a})|^\theta\leq\|\nabla f(\mathbf{x})\|
\end{equation}
for any $\mathbf{x}\in U$
\end{proposition}

By using Proposition \ref{KL2} and the finite subcover property of compact set,
we have the following proposition \cite{bolte,haraux}.
\begin{proposition}\label{KL3}
Let $E$ be the set of critical points of $f$, and $\Gamma$ be compact and connected subset of $E$.
If $f$ is a real-analytic function on $\mathbb{R}^n$ and $\mathbf{a}\in\Gamma$, then \\
(i) for any $\mathbf{b}\in\Gamma$, $f(\mathbf{b})=f(\mathbf{a})\triangleq\underline{f}$.\\
(ii) there is a neighborhood $U$ of $\Gamma$ and a real number $0<\theta<1$ such that
\begin{equation}
\forall \mathbf{x}\in U, |f(\mathbf{x})-\underline{f}|^\theta\leq\|\nabla f(\mathbf{x})\|.
\end{equation}
\end{proposition}

In the RALS algorithm,
the residual function $f(\mathbf{x})=f(\mathbf{A},\mathbf{B},\mathbf{C})$ is a polynomial function on
$\mathscr{X}=\mathbb{R}^{I\times r}\times\mathbb{R}^{J\times r}\times\mathbb{R}^{K\times r}$.
So it is also a real-analytic function on $\mathscr{X}$.
Unlike the work of Li et. al. \cite{li} showing that every limit point is a critical point,
the following theorem points out the global convergence of the RALS algorithm. Its proof is based on the Kurdyka-{\L}ojasiewicz inequality and
the proximal alternating minimum technique \cite{attouch,attouch1,bolte}.

\begin{theorem}
Let $\{\mathbf{x}^{(n)}\}_{n\in\mathbb{N}}$ be the sequence generated by the RALS algorithm.
If the sequence $\{\mathbf{x}^{(n)}\}_{n\in\mathbb{N}}$ is bounded,
this sequence converges to a critical point $\mathbf{x}^*$ of $f(\mathbf{x})$.
\end{theorem}
\begin{proof}
In the RALS algorithm, the residual function $f(\mathbf{x})=f(\mathbf{A},\mathbf{B},\mathbf{C})=
\frac{1}{2}\|\mathcal{T}-\sum\limits_{s=1}^r\mathbf{a}_s\circ\mathbf{b}_s\circ\mathbf{c}_s\|^2$ is a polynomial function on
$\mathscr{X}=\mathbb{R}^{I\times r}\times\mathbb{R}^{J\times r}\times\mathbb{R}^{K\times r}$,
where $\mathbf{a}_s,\mathbf{b}_s,\mathbf{c}_s$ are columns of $\mathbf{A,B}$ and $\mathbf{C}$ respectively.
From (\ref{RALS}),
we know that
\begin{equation}\label{decspeed}
f(\mathbf{x}^{(n)})-f(\mathbf{x}^{(n+1)})\geq\frac{1}{2}\lambda\|\mathbf{x}^{(n+1)}-\mathbf{x}^{(n)}\|^2
\end{equation}
and
\begin{equation}
\label{gradient}
\begin{aligned}
\nabla_{\mathbf{A}} f(\mathbf{A}^{(n+1)},\mathbf{B}^{(n)},\mathbf{C}^{(n)})+\lambda(\mathbf{A}^{(n+1)}-\mathbf{A}^{(n)})=0,\\
\nabla_{\mathbf{B}} f(\mathbf{A}^{(n+1)},\mathbf{B}^{(n+1)},\mathbf{C}^{(n)})+\lambda(\mathbf{B}^{(n+1)}-\mathbf{B}^{(n)})=0,\\
\nabla_{\mathbf{C}} f(\mathbf{A}^{(n+1)},\mathbf{B}^{(n+1)},\mathbf{C}^{(n+1)})+\lambda(\mathbf{C}^{(n+1)}-\mathbf{C}^{(n)})=0.\\
\end{aligned}
\end{equation}
From (\ref{decspeed}), we have that $\lim\limits_{n\rightarrow\infty}\|\mathbf{x}^{(n+1)}-\mathbf{x}^{(n)}\|=0$
and $\{f(\mathbf{x}^{(n)})\}_{n\in\mathbb{N}}$ is a monotonically decreasing sequence.
Let $\underline{f}=\lim\limits_{n\rightarrow\infty}f(\mathbf{x}^{(n)})$.

Due to the boundedness of $\{\mathbf{x}^{(n)}\}_{n\in\mathbb{N}}$,
the first equality in (\ref{gradient}) and the differentiability of $f(\mathbf{x})$,
there exist constants $\lambda_1,\lambda_2>0$ and $\mu_1>0$ such that
\begin{eqnarray*}
\|\nabla_{\mathbf{A}} f(\mathbf{A}^{(n+1)},\mathbf{B}^{(n+1)},\mathbf{C}^{(n+1)})\|_F &\leq&\|\nabla_{\mathbf{A}} f(\mathbf{A}^{(n+1)},\mathbf{B}^{(n+1)},\mathbf{C}^{(n+1)})-\nabla_{\mathbf{A}} f(\mathbf{A}^{(n+1)},\mathbf{B}^{(n)},\mathbf{C}^{(n)})\|_F \\
&&+\|\nabla_{\mathbf{A}} f(\mathbf{A}^{(n+1)},\mathbf{B}^{(n)},\mathbf{C}^{(n)})\|_F\\
&\leq&\lambda_1\|\mathbf{B}^{(n+1)}-\mathbf{B}^{(n)}\|_F+\lambda_2\|\mathbf{C}^{(n+1)}-\mathbf{C}^{(n)}\|_F+\lambda\|\mathbf{A}^{(n+1)}-\mathbf{A}^{(n)}\|_F\\
&\leq&\mu_1\|\mathbf{x}^{(n+1)}-\mathbf{x}^{(n)}\|
\end{eqnarray*}
for any $n\in\mathbb{N}$.
Similarly, there exist constants $\mu_2,\mu_3>0$ such that
\begin{eqnarray*}
\|\nabla_{\mathbf{B}} f(\mathbf{A}^{(n+1)},\mathbf{B}^{(n+1)},\mathbf{C}^{(n+1)})\|_F\leq\mu_2\|\mathbf{x}^{(n+1)}-\mathbf{x}^{(n)}\|\\
\|\nabla_{\mathbf{C}} f(\mathbf{A}^{(n+1)},\mathbf{B}^{(n+1)},\mathbf{C}^{(n+1)})\|_F\leq\mu_3\|\mathbf{x}^{(n+1)}-\mathbf{x}^{(n)}\|.
\end{eqnarray*}
It follows that there exists a constant $d>0$ such that
\begin{equation}\label{gradientcontrol}
\|\nabla_{\mathbf{x}}f(\mathbf{x}^{(n+1)})\|\leq d\|\mathbf{x}^{(n+1)}-\mathbf{x}^{(n)}\|
\end{equation}
for any $n\in\mathbb{N}$.

Denote the limit point set of $\{\mathbf{x}^{(n)}\}_{n\in\mathbb{N}}$ by $L$.
From the inequality (\ref{gradientcontrol}), any point in $L$ is a critical point of $f$.
It can be also checked that $L$ is a compact and connected set since
$\{\mathbf{x}^{(n)}\}_{n\in\mathbb{N}}$ is bounded and
$\lim\limits_{n\rightarrow\infty}\|\mathbf{x}^{(n+1)}-\mathbf{x}^{(n)}\|=0$.
So from Proposition \ref{KL3}, we have $f(\mathbf{x})=\underline{f}$ for any $\mathbf{x}\in L$,
and there is a neighborhood $U$ of $L$
and a real number $0<\theta<1$ such that
$|f(\mathbf{x})-\underline{f}|^\theta\leq\|\nabla f(\mathbf{x})\|$ for any $\mathbf{x}\in U$.
Since $L$ is the limit point set of $\{\mathbf{x}^{(n)}\}_{n\in\mathbb{N}}$,
it follows that $\mathbf{x}^{(n)}\in U$ when $n$ is large enough.
So there exists a positive integer $l$ such that
$|f(\mathbf{x}^{(n)})-\underline{f}|^\theta\leq\|\nabla f(\mathbf{x}^{(n)})\|$ when $n\geq l$.

Since the concavity of function $g(y)=(y-\underline{f})^{1-\theta}$ for some $0<\theta<1$ when $y\geq\underline{f}$,
\begin{equation}
\frac{(f(\mathbf{x}^{(n)})-\underline{f})^{1-\theta}-(f(\mathbf{x}^{(n+1)})-\underline{f})^{1-\theta}}{f(\mathbf{x}^{(n)})-f(\mathbf{x}^{(n+1)})}
\geq(1-\theta)(f(\mathbf{x}^{(n)})-\underline{f})^{-\theta}.
\end{equation}
Since $f(\mathbf{x}^{(n)})-f(\mathbf{x}^{(n+1)})\geq\frac{1}{2}\lambda\|\mathbf{x}^{(n+1)}-\mathbf{x}^{(n)}\|^2$ and $(f(\mathbf{x}^{(n)})-\underline{f})^{\theta}\leq\|\nabla f(\mathbf{x}^{(n)})\|\leq d\|\mathbf{x}^{(n)}-\mathbf{x}^{(n-1)}\|$,
we have that
\begin{equation}
\frac{2d((f(\mathbf{x}^{(n)})-\underline{f})^{1-\theta}-(f(\mathbf{x}^{(n+1)})-\underline{f})^{1-\theta})}{(1-\theta)\lambda}\geq\frac{\|\mathbf{x}^{(n+1)}-\mathbf{x}^{(n)}\|^2}{\|\mathbf{x}^{(n)}-\mathbf{x}^{(n-1)}\|}.
\end{equation}
Denote $\frac{2d((f(\mathbf{x}^{(n)})-\underline{f})^{1-\theta}-(f(\mathbf{x}^{(m)})-\underline{f})^{1-\theta})}{(1-\theta)\lambda}$
by $e_{n,m}$ where $m\geq n$.
So $\|\mathbf{x}^{(n+1)}-\mathbf{x}^{(n)}\|^2\leq\|\mathbf{x}^{(n)}-\mathbf{x}^{(n-1)}\|e_{n,n+1}$.
Moreover, $2\|\mathbf{x}^{(n+1)}-\mathbf{x}^{(n)}\|\leq\|\mathbf{x}^{(n)}-\mathbf{x}^{(n-1)}\|+e_{n,n+1}$. Thus,
\begin{eqnarray*}
2\sum\limits_{n=l}^k\|\mathbf{x}^{(n+1)}-\mathbf{x}^{(n)}\|&\leq &\sum\limits_{n=l}^k\|\mathbf{x}^{(n)}-\mathbf{x}^{(n-1)}\|+\sum\limits_{n=l}^ke_{n,n+1}\\
&\leq &\sum\limits_{n=l}^k\|\mathbf{x}^{(n+1)}-\mathbf{x}^{(n)}\|+\|\mathbf{x}^{(l)}-\mathbf{x}^{(l-1)}\|+\sum\limits_{n=l}^ke_{n,n+1}\\
&= &\sum\limits_{n=l}^k\|\mathbf{x}^{(n+1)}-\mathbf{x}^{(n)}\|+\|\mathbf{x}^{(l)}-\mathbf{x}^{(l-1)}\|+e_{l,k+1}
\end{eqnarray*}
So $\sum\limits_{n=l}^k\|\mathbf{x}^{(n+1)}-\mathbf{x}^{(n)}\|\leq\|\mathbf{x}^{(l)}-\mathbf{x}^{(l-1)}\|+e_{l,k+1}$.
Since $\lim\limits_{n\rightarrow\infty}\|\mathbf{x}^{(n+1)}-\mathbf{x}^{(n)}\|=0$ and $e_{l,k+1}$ is bounded for any $k\geq l$,
$\{\mathbf{x}^{(n)}\}_{n\in\mathbb{N}}$ is a Cauchy sequence.
So $\lim\limits_{n\rightarrow\infty}\mathbf{x}^{(n)}=\mathbf{x}^*$ and $\nabla f(\mathbf{x}^*)=0$.
\end{proof}

The proof here can also be shown by using the techniques in \cite{absil}
since the RALS algorithm satisfies the strong descent conditions of analytic cost functions.
As shown in \cite{attouch,bolte},
the global convergence rate can be further discussed regarding the value of $\theta$.
In particular, $\theta\in(0,1/2]$ gives a linear global convergent rate while $\theta\in(1/2,1)$ leads to a sublinear one.
But there is no further information on the specific value of $\theta$ for the residual function of the RALS algorithm.
In Appendix A, we discuss the local convergence rate of RALS and
show that when the sequence is close enough to the local minimum point,
the RALS algorithm has a linear local convergence rate.

\section{Conclusions and future outlook}
We discuss the convergence and acceleration of the regularized alternating least
square (RALS) algorithm for tensor approximation.
Under mild conditions, the RALS algorithm has a global convergence and a linear local convergence rate (see Appendix A).
As shown in the simulation experiments,
the accelerated versions of (R)ALS algorithm provide a faster speed compared to original ones.
Although the update map $T$ for the acceleration can also keep fixed points,
it still lacks of the theoretical guarantee on the effectiveness of acceleration.
Moreover, we would like to understand why a faster convergent rate can be obtained
by decreasing the regularization parameter to zero.
Furthermore, we are very interested in knowing if these convergence theories has any connection in generating swamps for tensor approximations.

\textbf{Acknowledgements.}
The authors are thankful to Hedy Attouch for some valuable suggestions on some references.

\appendix

\section{Local convergence rate of RALS}
First we introduce some basic properties of the update operator S defined in (\ref{SRALS}).
\begin{theorem}\label{fixpoint}
The operator $S$ is smooth on the space
$\mathscr{X}=\mathbb{R}^{I\times r}\times\mathbb{R}^{J\times r}\times\mathbb{R}^{K\times r}$.
If $\mathbf{x}^*=(\mathbf{A}^*,\mathbf{B}^*,\mathbf{C}^*)$ is a local minimum point of $f$,
$\mathbf{x}^*$ is a fixed point of $S$.
\end{theorem}
\begin{proof}
From the update mechanism (\ref{RALS}) and the exact expressions
(\ref{ESRALS}) for $\mathbf{A}^{(n+1)},\mathbf{B}^{(n+1)}$ and $\mathbf{C}^{(n+1)}$,
the update operator S is smooth on $\mathscr{X}=\mathbb{R}^{I\times r}\times\mathbb{R}^{J\times r}\times\mathbb{R}^{K\times r}$.

If $(\mathbf{A}^*,\mathbf{B}^*,\mathbf{C}^*)$ is a local minimum point of $f$,
we have that $\mathbf{A}^{(n+1)}=\mathbf{A}^*$ when $\mathbf{B}^{(n)}=\mathbf{B}^*,\mathbf{C}^{(n)}=\mathbf{C}^*$.
Since $f(\mathbf{A},\mathbf{B}^*,\mathbf{C}^*)+\frac{1}{2}\lambda\|\mathbf{A}-\mathbf{A}^*\|^2$
is a strict convex function in $\mathbf{A}$, then
$$f(\mathbf{A}^{(n+1)},\mathbf{B}^*,\mathbf{C}^*)+\frac{1}{2}\lambda\|\mathbf{A}^{(n+1)}-\mathbf{A}^*\|^2<f(\mathbf{A}^*,\mathbf{B}^*,\mathbf{C}^*)$$
from the update mechanism shown in (\ref{RALS}) if $\mathbf{A}^{(n+1)}\neq\mathbf{A}^*$.
So $f(\mathbf{A}^{(n+1)},\mathbf{B}^*,\mathbf{C}^*)<f(\mathbf{A}^*,\mathbf{B}^*,\mathbf{C}^*)$.
Since $f$ is a convex function in $\mathbf{A}$ when fixing $\mathbf{B,C}$,
it follows that
$$f(a\mathbf{A}^{(n+1)}+(1-a)\mathbf{A}^*,\mathbf{B}^*,\mathbf{C}^*)\leq af(\mathbf{A}^{(n+1)},\mathbf{B}^*,\mathbf{C}^*)+(1-a)f(\mathbf{A}^*,\mathbf{B}^*,\mathbf{C}^*)$$
for any $a\in(0,1)$.
Thus, $f(a\mathbf{A}^{(n+1)}+(1-a)\mathbf{A}^*,\mathbf{B}^*,\mathbf{C}^*)<f(\mathbf{A}^*,\mathbf{B}^*,\mathbf{C}^*)$ for any $a\in(0,1)$,which contradicts with the fact that $(\mathbf{A}^*,\mathbf{B}^*,\mathbf{C}^*)$ is a local minimum of $f$.
So if $(\mathbf{A}^*,\mathbf{B}^*,\mathbf{C}^*)$ is a local minimum point of $f$,
we have that $\mathbf{A}^{(n+1)}=\mathbf{A}^*$ when $\mathbf{B}^{(n)}=\mathbf{B}^*,\mathbf{C}^{(n)}=\mathbf{C}^*$.
Furthermore, it follows that $(\mathbf{A}^*,\mathbf{B}^*,\mathbf{C}^*)=S(\mathbf{A}^*,\mathbf{B}^*,\mathbf{C}^*)$ from (\ref{RALS}). Thus, a local minimum point $(\mathbf{A}^*,\mathbf{B}^*,\mathbf{C}^*)$ of $f$ is a fixed point of $S$.
\end{proof}

Next, we will discuss about the contractive property of the operator $S$
under the framework of iterative solution of nonlinear equations \cite{ortega}.
A similar approach \cite{rohwedder,uschmajew} has been applied on the ALS algorithm as well as on the alternating linear scheme for tensor train format \cite{holtz}.

Any point $\mathbf{x}=(\mathbf{A},\mathbf{B},\mathbf{C})\in\mathbb{R}^{I\times r}\times\mathbb{R}^{J\times r}\times\mathbb{R}^{K\times r}$
can be viewed as a vector $\mathbf{x}=(\mathbf{x}_A^T,\mathbf{x}_B^T,\mathbf{x}_C^T)^T$,
where $\mathbf{x}_A\in\mathbb{R}^{rI},\mathbf{x}_B\in\mathbb{R}^{rJ},\mathbf{x}_C\in\mathbb{R}^{rK}$
are the vectorized form (column stacked) of $\mathbf{A},\mathbf{B},\mathbf{C}$, respectively.
Denote the vector value function, $\frac{\partial f(\mathbf{x}_A,\mathbf{y}_B,\mathbf{y}_C)}{\partial{\mathbf{x}_A}}+\lambda(\mathbf{x}_A-\mathbf{y}_A)$,
by $g_A(\mathbf{x},\mathbf{y})$ where $\mathbf{y}=(\mathbf{y}_A^T,\mathbf{y}_B^T,\mathbf{y}_C^T)^T$ and $\mathbf{y}_A\in\mathbb{R}^{rI},\mathbf{y}_B\in\mathbb{R}^{rJ},\mathbf{y}_C\in\mathbb{R}^{rK}$.
Similarly, denote $\frac{\partial f(\mathbf{x}_A,\mathbf{x}_B,\mathbf{y}_C)}{\partial\mathbf{x}_B}+\lambda(\mathbf{x}_B-\mathbf{y}_B)$
by $g_B(\mathbf{x},\mathbf{y})$,
and $\frac{\partial f(\mathbf{x}_A,\mathbf{x}_B,\mathbf{x}_C)}{\partial\mathbf{x}_C}+\lambda(\mathbf{x}_C-\mathbf{y}_C)$
by $g_C(\mathbf{x},\mathbf{y})$.
Denote the vector value function $(g_A^T(\mathbf{x},\mathbf{y}),g_B^T(\mathbf{x},\mathbf{y}),g_C^T(\mathbf{x},\mathbf{y}))^T$
by $G(\mathbf{x},\mathbf{y})$.
From the equations in (\ref{gradient}), we know that $G(\mathbf{x}^{(n+1)},\mathbf{x}^{(n)})=0$.

Let $\mathbf{x}^*$ be a local minimizer of the residual function $f$.
Since $f$ is twice continuously differentiable function,
the Hessian matrix $\mathbf{H}=\frac{\partial^2 f(\mathbf{x}^*)}{\partial\mathbf{x}\partial\mathbf{x}}$ of $f$ at $\mathbf{x}^*$
is positive semidefinite and it has nine block matrices corresponding to $\mathbf{A,B,C}$.
From direct computation,
the matrix $\frac{\partial G(\mathbf{x}^*,\mathbf{x}^*)}{\partial\mathbf{x}}$ is the lower triangular block matrix of $\mathbf{H}$ with an additional $\lambda\mathbf{I}$ on the diagonal blocks and the matrix $\frac{\partial G(\mathbf{x}^*,\mathbf{x}^*)}{\partial\mathbf{y}}$ is the strict upper block matrix of
$\mathbf{H}$ minus $\lambda\mathbf{I}$,
where $\mathbf{I}$ is an identity matrix in $\mathbb{R}^{r(I+J+K)\times r(I+J+K)}$.
They are
\[ \mathbf{H}= \left(
\begin{array}{clr}
\frac{\partial^2 f(\mathbf{x}^*)}{\partial \mathbf{x}_A\partial \mathbf{x}_A} & \frac{\partial^2 f(\mathbf{x}^*)}{\partial \mathbf{x}_B\partial \mathbf{x}_A} & \frac{\partial^2 f(\mathbf{x}^*)}{\partial \mathbf{x}_C\partial \mathbf{x}_A}\\
\frac{\partial^2 f(\mathbf{x}^*)}{\partial \mathbf{x}_A\partial \mathbf{x}_B} & \frac{\partial^2 f(\mathbf{x}^*)}{\partial \mathbf{x}_B\partial \mathbf{x}_B} & \frac{\partial^2 f(\mathbf{x}^*)}{\partial \mathbf{x}_C\partial \mathbf{x}_B}\\
\frac{\partial^2 f(\mathbf{x}^*)}{\partial \mathbf{x}_A\partial \mathbf{x}_C} & \frac{\partial^2 f(\mathbf{x}^*)}{\partial \mathbf{x}_B\partial \mathbf{x}_C} & \frac{\partial^2 f(\mathbf{x}^*)}{\partial \mathbf{x}_C\partial \mathbf{x}_C}\\
\end{array} \right) \],

\[ \frac{\partial G(\mathbf{x}^*,\mathbf{x}^*)}{\partial\mathbf{x}}= \left(
\begin{array}{clr}
\frac{\partial^2 f(\mathbf{x}^*)}{\partial \mathbf{x}_A\partial \mathbf{x}_A}+\lambda\mathbf{I}_A & \mathbf{0} & \mathbf{0}\\
\frac{\partial^2 f(\mathbf{x}^*)}{\partial \mathbf{x}_A\partial \mathbf{x}_B} & \frac{\partial^2 f(\mathbf{x}^*)}{\partial \mathbf{x}_B\partial \mathbf{x}_B}+\lambda\mathbf{I}_B & \mathbf{0}\\
\frac{\partial^2 f(\mathbf{x}^*)}{\partial \mathbf{x}_A\partial \mathbf{x}_C} & \frac{\partial^2 f(\mathbf{x}^*)}{\partial \mathbf{x}_B\partial \mathbf{x}_C} & \frac{\partial^2 f(\mathbf{x}^*)}{\partial \mathbf{x}_C\partial \mathbf{x}_C}+\lambda\mathbf{I}_C\\
\end{array} \right) \],

\[ \frac{\partial G(\mathbf{x}^*,\mathbf{x}^*)}{\partial\mathbf{y}}= \left(
\begin{array}{clr}
-\lambda\mathbf{I}_A & \frac{\partial^2 f(\mathbf{x}^*)}{\partial \mathbf{x}_B\partial \mathbf{x}_A} &  \frac{\partial^2 f(\mathbf{x}^*)}{\partial \mathbf{x}_C\partial \mathbf{x}_A}\\
\mathbf{0} & -\lambda\mathbf{I}_B & \frac{\partial^2 f(\mathbf{x}^*)}{\partial \mathbf{x}_C\partial \mathbf{x}_B}\\
\mathbf{0} & \mathbf{0} & -\lambda\mathbf{I}_C\\
\end{array} \right) \]
where $\mathbf{I}_A,\mathbf{I}_B,\mathbf{I}_C$ are identity matrices in $\mathbb{R}^{rI\times rI},\mathbb{R}^{rJ\times rJ},\mathbb{R}^{rK\times rK}$, respectively.

The matrix $\frac{\partial G(\mathbf{x}^*,\mathbf{x}^*)}{\partial\mathbf{x}}$ is nonsingular
since all the three diagonal blocks of $\mathbf{H}$ are positive semidefinite.
The Hessian matrix $\mathbf{H}$ can be rewritten into $\mathbf{D}-\mathbf{L}-\mathbf{U}$,
where $\mathbf{D}$ is a diagonal block matrix, $-\mathbf{L}$ is a strict lower block matrix
and $-\mathbf{U}$ is a strict upper block matrix of $\mathbf{H}$.
Thus we have that
\begin{eqnarray*}
-\frac{\partial G(\mathbf{x}^*,\mathbf{x}^*)}{\partial\mathbf{x}}^{-1}\frac{\partial G(\mathbf{x}^*,\mathbf{x}^*)}{\partial\mathbf{y}}
&=&(\lambda\mathbf{I}+\mathbf{D}-\mathbf{L})^{-1}(\lambda\mathbf{I}+\mathbf{U})\\
&=&I-(\lambda\mathbf{I}+\mathbf{D}-\mathbf{L})^{-1}(\mathbf{D}-\mathbf{L}-\mathbf{U}).
\end{eqnarray*}

Let $\mathbf{M}=\lambda\mathbf{I}+\mathbf{D}-\mathbf{L}$.
From Theorem 3.2 in \cite{lee}, since $\mathbf{M}+\mathbf{M}^T-\mathbf{H}$ is positive definite,
thus $\|\mathbf{I}-\mathbf{M}^{-1}\mathbf{H}\|_{\mathbf{H}}=\max\limits_{\|\mathbf{x}\|_{\mathbf{H}}\neq 0}\frac{\|(\mathbf{I}-\mathbf{M}^{-1}\mathbf{H})\mathbf{x}\|_{\mathbf{H}}}{\|\mathbf{x}\|_{\mathbf{H}}}<1$,
where $\|\mathbf{y}\|_{\mathbf{H}}=(\mathbf{y}^T\mathbf{H}\mathbf{y})^{\frac{1}{2}}$ is a seminorm on $\mathbf{y}$.
If we further assume that $\mathbf{H}$ is a positive definite matrix,
$\|\mathbf{y}\|_{\mathbf{H}}$ is a norm on $\mathbf{y}$ and $\|\mathbf{I}-\mathbf{M}^{-1}\mathbf{H}\|_{\mathbf{H}}$ is
a matrix norm on $\mathbf{I}-\mathbf{M}^{-1}\mathbf{H}$.

Since $\mathbf{x}^*$ is a local minimum point of $f$,
we have that $\mathbf{x}^*$ is a fixed point of $S$ from Theorem \ref{fixpoint}.
Furthermore, it follows that $G(\mathbf{x}^*,\mathbf{x}^*)=0$ by equations in (\ref{gradient}).
Then from the implicit function theorem, there is a neighborhood $U$ of $\mathbf{x}^*$
such that $\mathbf{x}=S(\mathbf{y})$ when $\mathbf{y}\in U$ and $S'(\mathbf{x}^*)=\mathbf{I}-\mathbf{M}^{-1}\mathbf{H}$.
Since $S'(\mathbf{x}^*)=\mathbf{I}-\mathbf{M}^{-1}\mathbf{H}$ and $\|S'(\mathbf{x}^*)\|_{\mathbf{H}}<q<1$,
there exists a small enough neighborhood $V$ of $\mathbf{x}^*$ such that $\|S'(\mathbf{y})\|_{\mathbf{H}}<q$ for $\mathbf{y}\in V$.
So there exists a sufficiently small neighborhood $W$ of $\mathbf{x}^*$ such that $S(\mathbf{y})\in W$, $\|S'(\mathbf{y})\|_{\mathbf{H}}<q$
and $\|S(\mathbf{y})-S(\mathbf{x}^*)\|_{\mathbf{H}}<q\|\mathbf{y}-\mathbf{x}^*\|_{\mathbf{H}}$ for $\mathbf{y}\in W$.
So if $\mathbf{x}^{(n)}\in W$ for some $n\in\mathbb{N}$,
then $\mathbf{x}^{(n+1)}\in W$, $\|\mathbf{x}^{(n+1)}-\mathbf{x}^*\|_{\mathbf{H}}<q\|\mathbf{x}^{(n)}-\mathbf{x}^*\|_{\mathbf{H}}$
and $\lim\limits_{n\rightarrow\infty}\mathbf{x}^{(n)}=\mathbf{x}^*$.
Furthermore, if $\mathbf{x}^{(0)}\in W$, we can obtain that $\limsup\limits_{n\rightarrow\infty}\|\mathbf{x}^{(n)}-\mathbf{x}^*\|^{\frac{1}{n}}\leq q$
and $\limsup\limits_{n\rightarrow\infty}\|f(\mathbf{x}^{(n)})-f(\mathbf{x}^*)\|^{\frac{1}{n}}\leq q$
from the equivalence of norms in the finite dimensional space.
So we obtain that the RALS algorithm has linear local convergence rate when $\mathbf{x}^n$ is enough close to a local minimum point $\mathbf{x}^*$
and the Hessian matrix $\frac{\partial^2 f(\mathbf{x}^*)}{\partial\mathbf{x}\partial\mathbf{x}}$ of $f$ at $\mathbf{x}^*$ is positive definite.
\begin{theorem}
Let $\{\mathbf{x}^{(n)}\}_{n\in\mathbb{N}}$ be the sequence generated by the RALS algorithm.
Assume that $\mathbf{x}^*$ is a local minimum point of $f$ and
the Hessian matrix $H=\frac{\partial^2 f(\mathbf{x}^*)}{\partial\mathbf{x}\partial\mathbf{x}}$ is positive definite.
There exist a neighborhood $W$ of $\mathbf{x}^*$ and a positive constant $q<1$
such that:\\
$(i)$ if $\mathbf{x}^{(n)}\in W$ for some $n\in\mathbb{N}$,
then $\mathbf{x}^{(n+1)}\in W$,
$\|\mathbf{x}^{(n+1)}-\mathbf{x}^*\|_{\mathbf{H}}<q\|\mathbf{x}^{(n)}-\mathbf{x}^*\|_{\mathbf{H}}$
and $\lim\limits_{n\rightarrow\infty}\mathbf{x}^{(n)}=\mathbf{x}^*$.\\
$(ii)$ if $\mathbf{x}^{(0)}\in W$, then $\limsup\limits_{n\rightarrow\infty}\|\mathbf{x}^{(n)}-\mathbf{x}^*\|^{\frac{1}{n}}\leq q$
and $\limsup\limits_{n\rightarrow\infty}\|f(\mathbf{x}^{(n)})-f(\mathbf{x}^*)\|^{\frac{1}{n}}\leq q$.
\end{theorem}

In the work of Uschmajew \cite{uschmajew}, a similar result was provided for the ALS algorithm with the objective function,
$g_\lambda(\mathbf{A},\mathbf{B},\mathbf{C})=f(\mathbf{A},\mathbf{B},\mathbf{C})+
\lambda(\|\mathbf{A}\|^2+\|\mathbf{B}\|^2+\|\mathbf{C}\|^2)$. A natural positive definite property of $\frac{\partial^2 g_\lambda(\mathbf{x}^*)}{\partial\mathbf{x}\partial\mathbf{x}}$ with large enough $\lambda$ can guarantee the linearly convergent rate.

%
%

\end{document}